\documentclass[11pt]{amsart}

\usepackage{graphicx}
\usepackage{color}

\usepackage{amsfonts,amsmath,latexsym,amssymb,amsthm}
\usepackage{hyperref}
\usepackage{geometry}
\usepackage{txfonts}
\usepackage{enumerate}

\newtheorem{theorem}{Theorem}

\newtheorem{lemma}{Lemma}
\theoremstyle{remark}

\begin{document}

\title{Magnetic field influence on the discrete spectrum of locally deformed leaky wires}

\author[D. Barseghyan]{Diana Barseghyan}
\address{Department of Mathematics, Faculty of Science, University of Ostrava, 30.~dubna 22, 70103 Ostrava, Czech Republic \\
E-mail: {diana.barseghyan@osu.cz}
}

\author[P. Exner]{Pavel Exner}
\address{Doppler Institute for Mathematical Physics and Applied Mathematics\\
Czech Technical University in Prague\\ B\v{r}ehov\'{a} 7, 11519 Prague, Czech Republic,
{\rm and}
Department of Theoretical Physics\\
Nuclear Physics Institute, Czech Academy of Sciences,
25068 \v{R}e\v{z}, Czech Republic\\
E-mail: {exner@ujf.cas.cz}
}

\keywords{Singular Schr\"odinger operators, magnetic field, discrete spectrum}
\subjclass[2010]{35J15; 35P15; 81Q10}

\maketitle

\begin{abstract}
We consider magnetic Schr\"odinger operator $H=(i \nabla +A)^2-\alpha \delta_\Gamma$ with an attractive singular interaction supported by a piecewise smooth curve $\Gamma$ being a local deformation of a straight line. The magnetic field $B$ is supposed to be nonzero and local. We show that the essential spectrum is $[-\frac14\alpha^2,\infty)$, as for the non-magnetic operator with a straight $\Gamma$, and demonstrate a sufficient condition for the discrete spectrum of $H$ to be empty.
\end{abstract}

\section{Introduction}
\label{s:intro}

It is well known that a magnetic field, even a local one, can influence substantially the behavior of waveguide systems, in particular their geometrically induced discrete spectrum. While a particle confined to a fixed-profile tube with Dirichlet boundary can exists in localized states whenever the tube is bent or locally deformed \cite{EK15}, the presence of a local magnetic field can destroy such a discrete spectrum. This is a consequence of a Hardy-type inequality\footnote{Note the the presence of a magnetic field is not the only effect that may cause an operator subcriticality, for another example see, e.g., \cite{BEK05}.} proved by Ekholm and Kova\v{r}\'{\i}k \cite{EK05}, for a more general form of this result see \cite[Thm~9.9]{Ra17}.

The analogous effect of bound state existence coming from the geometry of the interaction support was observed is a class of singular Schr\"odinger operators, usually dubbed \emph{leaky quantum wires}, with attractive contact interaction supported by a curve \cite[Chap.~10]{EK15}. Here too, as observed first in \cite{EI01}, perturbation of a straight line by a bend or local deformation induces, under suitable asymptotic straightness conditions, the existence of a nonempty discrete spectrum. The above waveguide result then leads to the natural question what would happen with such bound states in presence of a magnetic field. The aim of this letter is to provide a partial answer.

Let us mention that magnetic Schr\"odinger operators with such singular interactions have been studied recently, however, in a different situation where the field was homogeneous. In such a case that spectral picture changes completely. The magnetic field itself turns the essential spectrum into the family of infinitely degenerate Landau levels and the effect of the singular interaction depends on its support: a finite curve splits the Landau levels into clusters of eigenvalues \cite{BEHL20}, a straight line produces an absolutely continuous spectrum analogous to the well-known edges states \cite{DP99} or the Iwatsuka model \cite{Iw85, MP97}, and a non-straight support may give rise to isolated eigenvalues \cite{ELPO18}.

A local field is a much weaker perturbation which does not change the essential spectrum, the question is what has to happen so that the discrete spectrum would be destroyed. We are going to prove one such sufficient condition referring to the situation when the curve supporting the singular interaction is a local deformation of a straight line. To be more specific, we are going to consider operator that can be formally written as
\begin{subequations}
\label{Hamiltonian}
\begin{equation}\label{Hformal}
H=(i \nabla +A)^2-\alpha \delta_\Gamma,
\end{equation}
where  $A$ is a vector potential corresponding to  magnetic field $B$, $\alpha>0$, and $\Gamma$ is a curve of the form $\Gamma=(x, \gamma(x))$, where $\gamma(\cdot)$ a continuous, piecewise $C^1$ smooth function. There are several ways to define the operator properly, the simplest one is to identify it with the unique self-adjoint operator associated with the quadratic form
\begin{equation} \label{Hform}
q:\: q[\psi] = \int_{\mathbb{R}^2} |i\nabla+A|^2 \mathrm{d}x\mathrm{d}y -\alpha\int_{\Gamma} |\psi|^2 \mathrm{d}\mu, \quad \psi\in \mathcal{H}^1(\mathbb{R}^2),
\end{equation}
\end{subequations}
where $\mu$ is the measure on $\Gamma$ referring to the arc length of the curve. In addition to the stated regularity of $\Gamma$, we will assume that
\begin{enumerate}[(a)]
\setlength{\itemsep}{3pt}
\item the magnetic field $B=\mathrm{rot}A$ is compactly supported, $\mathrm{supp}(B)\subset\Omega:=(-a,a)^2$ for some $a>0$,
\item outside $\Omega$ the curve $\Gamma$ coincides with the straight line $\Gamma_0:=\{x, x\}_{x\in\mathbb{R}}$ and for $x\in (-a,a)$ it stays within $\Omega$, that is, $|\gamma(x)|\le a$.
\end{enumerate}
Let us note that if we rotate the axes passing to the coordinates $x\pm y$, the segment of $\Gamma$ in $\Omega$ may or may not be graph of a function. It is straightforward to check that under the stated assumption the form is closed and bounded from below, hence the operator $H$ is well defined. The operator domain of $H$ is the same as for the operator without the magnetic field consisting of functions $\psi\in \mathcal{H}^1(\mathbb{R}^2) \cap \mathcal{H}^2(\mathbb{R}^2 \setminus\Gamma)$ with the property that the normal derivative of $\psi$ at the point $x\in\Gamma$ has the jump equal to $-\alpha\psi(x)$.

\section{Main results}
\label{main}
\setcounter{equation}{0}
It is expected that the two local perturbations, the magnetic field and the departure of $\Gamma$ from the straight line will not alter the essential spectrum. This is indeed the case:
\begin{theorem}\label{th.1}
Under the stated assumptions the essential spectrum of operator $H$ coincides with the half-line $\big[-\frac14\alpha^2, \infty\big)$.
\end{theorem}

Our main interest concerns the possibility that the applied magnetic field destroys the discrete spectrum. The main result of this letter is the following sufficient condition:
\begin{theorem}\label{th.2}
Adopt the same assumptions and suppose that $B$ is not identically zero. Then there exists a constant $\alpha_0=\alpha_0(\Omega,B)>0$ such that for $\alpha\in(0,\alpha_0]$ the spectrum of operator (\ref{Hformal}) below  $-\frac14\alpha^2$ is empty provided that $$\|\gamma'\|_\infty\le \frac{2 \alpha_0^2}{\alpha^2}-1.$$
\end{theorem}

\section{Proofs}
\label{proofs}
\setcounter{equation}{0}

Let us begin with the \emph{proof of Theorem~\ref{th.1}}. First we are going to show that the essential spectrum of $H$ contains the half-line $\big[-\frac14\alpha^2, \infty\big)$. We employ the Weyl criterion \cite[Thm~VII.12]{RS81}. Put $\lambda=-\frac14\alpha^2+2p^2,\,p\ne0$, and consider the sequence of vectors
$$
\psi_k:\: \psi_k(x, y)=\frac{1}{\sqrt{k}}\, \mathrm{e}^{-\frac{\alpha}{2\sqrt{2}}|x-y|}\, \mathrm{e}^{i p(x+y)/2}\, \chi\left(\frac{x}{k}\right),
$$
where $\chi$ is a smooth function with support in $(1, 2)$ and $k\in \mathbb{N}$.

It is straightforward to check that $\psi_k\in \mathrm{Dom}(H)$; the idea is to use the fact that there is a part of plane where the magnetic field has no influence. Indeed, choosing the Landau gauge for the vector potential $A$, that is, putting $A(x,y)=\Big(-\int_0^y B(x,t)\,\mathrm{d}t,0\Big)$, we get $A(x,y)=0$ for $|x|>a$.
For $k$ large enough we thus have
\begin{align*}
\frac{\partial\psi_k(x,y)}{\partial x} &= \frac{1}{\sqrt{k}}\, \mathrm{e}^{-\frac{\alpha}{2\sqrt{2}}|x-y|}\, \mathrm{e}^{i p(x+y)/2}\, \bigg[\bigg( -\frac{\alpha}{2\sqrt{2}}\mathrm{sgn}(x-y) + \frac{ip}{2} \bigg) \chi\left(\frac{x}{k}\right) + \frac{1}{k}\chi'\left(\frac{x}{k}\right) \bigg], \\
\frac{\partial\psi_k(x,y)}{\partial y} &= \frac{1}{\sqrt{k}}\, \mathrm{e}^{-\frac{\alpha}{2\sqrt{2}}|x-y|}\, \mathrm{e}^{i p(x+y)/2}\, \bigg(\frac{\alpha}{2\sqrt{2}}\mathrm{sgn}(x-y) + \frac{ip}{2} \bigg),
\end{align*}
hence using \eqref{Hform} and $\mathrm{d}\mu = \sqrt{2}\,\mathrm{d}x$ we get by a direct computation
\begin{align*}
q[\psi_k] - \Big(\frac{\alpha^2}{4} +p^2\Big)\|\psi_k\|^2 &= \Big(\frac{\alpha^2}{4}+p^2\Big)\,\frac{2\sqrt{2}}{\alpha}\,\|\chi\|^2 + \frac{1}{k^2} \,\frac{2\sqrt{2}}{\alpha}\,\|\chi'\|^2 -\alpha\sqrt{2}\,\|\chi\|^2 \\ & \; + \Big(\frac{\alpha^2}{4}-p^2\Big)\,\frac{2\sqrt{2}}{\alpha}\,\|\chi\|^2 = \frac{1}{k^2} \,\frac{2\sqrt{2}}{\alpha}\,\|\chi'\|^2 = \mathcal{O}(k^{-2}),
\end{align*}
and since  $\|\psi_k\|^2= \frac{2\sqrt{2}}{\alpha}\,\|\chi\|^2$ is independent of $k$, we infer that $-\frac14\alpha^2+p^2\in \sigma(H)$. Moreover, one can choose a sequence $\{k_n\}_{n=1}^\infty$ such that $k_n\to\infty$ as $n\to\infty$ and the supports of the functions $\psi_{k_n}$ are mutually disjoint which means that $-\frac14\alpha^2+p^2\in \sigma_\mathrm{ess}(H)$.

Next we have to establish that the spectrum of $H$ below $-\frac14\alpha^2$, if any, can be only discrete. Neumann bracketing yields the estimate
\begin{equation}
\label{bracketing}
H\ge H_1\oplus H_2,
\end{equation}
where $H_1$ is the Neumann restriction of $H$ to $L^2(\mathbb{R}^2\backslash \Omega)$ and $H_2$ is the complementary Neumann restriction to $L^2(\Omega)$. Denoting $\Gamma_\Omega:= \Gamma\upharpoonright\Omega$ and using the conventional abbreviations for partial derivatives, one can check easily that
\begin{align*}
& \int_{\mathbb{R}^2\backslash \Omega} |i \nabla u + A u|^2\,\mathrm{d}x\,\mathrm{d}y- \alpha \int_{\Gamma\backslash \Gamma_\Omega} |u|^2\,\mathrm{d}\mu \\ & \ge
\int_{\mathbb{R}\backslash (-a, a)} \int_{\mathbb{R}}|u_y|^2\,\mathrm{d}y\,\mathrm{d}x+ \int_{\mathbb{R}\backslash (-a, a)}\int_{\mathbb{R}}\Big|i u_x-u\! \int_0^y B(x, t)\,\mathrm{d}t\Big|^2\,\mathrm{d}x\,\mathrm{d}y-  \alpha \sqrt{2} \int_{\mathbb{R}\backslash (-a, a)}|u(x, x)|^2\,\mathrm{d}x
\\ &
\ge\int_{\mathbb{R}\backslash (-a, a)} \int_{\mathbb{R}}|u_y|^2\,\mathrm{d}y\,\mathrm{d}x- \frac{\alpha}{\sqrt{2}} \int_{\mathbb{R}\backslash (-a, a)}|u(x, x)|^2 \,\mathrm{d}x \\ & \quad + \int_{\mathbb{R}\backslash (-a, a)}\int_{\mathbb{R}}\Big|i u_x-u \int_0^y B(x, t)\,\mathrm{d}t\Big|^2\,\mathrm{d}x\,\mathrm{d}y-  \frac{\alpha}{\sqrt{2}} \int_{\mathbb{R}\backslash (-a, a)}|u(x, x)|^2\,\mathrm{d}x.
\end{align*}
In the last term on the right-hand side we may replace $\int_{\mathbb{R}\backslash (-a, a)}|u(x, x)|^2\,\mathrm{d}x$ by $\frac{\alpha}{\sqrt{2}} \int_{\mathbb{R}\backslash (-a, a)}|u(y, y)|^2\,\mathrm{d}y$. Noting then that
for any $y\in \mathbb{R}\setminus[-a,a]$ the integral $\int_0^y B(x,t)\,\mathrm{d}t$ depends only on $x$ and using the fact that the principal eigenvalue of operator $-\frac{\mathrm{d}^2}{\mathrm{d} t^2}- \frac{\alpha}{\sqrt{2}} \delta(t)$ is $-\frac18{\alpha^2}$, the above estimate implies
\begin{align}
\int_{\mathbb{R}^2\backslash \Omega} |i \nabla u + A u|^2\,\mathrm{d}x\,\mathrm{d}y- \alpha \int_{\Gamma\backslash \Gamma_\Omega} |u|^2\,\mathrm{d}\mu
& \ge-\frac{\alpha^2}{8}\, \bigg(\int_{\mathbb{R}\backslash (-a, a)}\int_{\mathbb{R}} + \int_{\mathbb{R}}\int_{\mathbb{R}\backslash (-a,a)}\bigg) |u|^2\,\mathrm{d}x\,\mathrm{d}y \nonumber \\ \label{final} & \ge-\frac{\alpha^2}{4}\int_{\mathbb{R}^2\backslash \Omega}|u|^2\,\mathrm{d}x\,\mathrm{d}y\,;
\end{align}
this means that the spectrum of $H_1$ below $-\frac14{\alpha^2}$ is empty.

It remains to deal with operator $H_2$. Our task will be done if we check that its spectrum is purely discrete; by minimax principle, it is sufficient to show that $H_2$ is bounded from below by an operator with a purely discrete spectrum. To this aim we will estimate the integral $\int_{\Gamma_\Omega} |u|^2\,\mathrm{d}\mu$ appearing in the singular perturbation using the magnetic Sobolev norm of $u$. To begin with, it is easy to see that for any fixed $x$ there is a $y_0=y_0(x)\in(-a, a)$ such that
$$
|u(x, y_0)|^2\le \frac{1}{2a} \int_{-a}^a |u(x, z)|^2 dz.
$$
In view of the identity
$$
|u(x,y)|^2 =\int_{y_0}^y \frac{\partial}{\partial z}|u(x,z)|^2\,\mathrm{d}z +|u(x,y_0)|^2,
$$
where the derivative in the integral equals $2\,\mathrm{Re}\big(\overline{u(x,z)}u_z(x,z)\big)$, we have the estimate
\begin{align*}
|u(x, y)|^2 & \le \varepsilon \int_{-a}^a |u_z(x, z)|^2\,dz+ \frac{1}{\varepsilon}\int_{-a}^a |u(x, z)|^2\,\mathrm{d}z +\frac{1}{\sqrt{2a}}\int_{-a}^a |u(x, z)|^2\,\mathrm{d}z\\ & \le \varepsilon \int_{-a}^a |u_z(x, z)|^2\,dz+\left(\frac{1}{\varepsilon}+ \frac{1}{\sqrt{2a}}\right)\int_{-a}^a |u(x, z)|^2\,\mathrm{d}z.
\end{align*}
for any $\varepsilon>0$. Choosing thus $A=\left(-\int_0^y B(x, t)\,dt, 0\right)$ we infer that
\begin{align*}
\int_{\Gamma_\Omega} |u|^2\, \mathrm{d}\mu &=  \int_{-a}^a  |u(x, \gamma(x))|^2\,\sqrt{1+\gamma'(x)^2}\,\mathrm{d}x \\ & \le  \sqrt{1+\|\gamma'\|^2_\infty}\,\left(\varepsilon \int_\Omega |i \nabla u+A u|^2\,\mathrm{d}x\,\mathrm{d}y+\left(\frac{1}{\varepsilon}+\frac{1}{\sqrt{2a}}\right) \int_\Omega |u|^2\,\mathrm{d}x\,\mathrm{d}y\right).
\end{align*}
This means that  $H_2$ is, in the sense of quadratic forms, estimated from below by the operator
$$
\left(1-\alpha \varepsilon\sqrt{1+\|\gamma'\|^2_\infty}\right)  \big(i \nabla u+A\big)^2- \alpha \left(\frac{1}{\varepsilon}+\frac{1}{\sqrt{2a}}\right)
$$
which for small enough values of $\varepsilon$ has a purely discrete spectrum. This establishes the claim of Theorem~\ref{th.1}.


\medskip

In order to \emph{prove Theorem~\ref{th.2}} we need the following two lemmata:

\begin{lemma}\label{1lemma}
Assume that $B\in L_{\mathrm{loc}}(\mathbb{R}^2)$ is not identically zero on a bounded domain $\omega\subset\mathbb{R}^2$, then the principal eigenvalue of the magnetic Neumann Laplacian on $\omega$ with the field $B$ is positive.
\end{lemma}
\begin{proof}
We have to demonstrate that
$$ 
\underset{\mathcal{H}^1(\omega),\,\|u\|=1}{\mathrm{inf}} \int_\omega |i \nabla u+ A u|^2\,\mathrm{d}x\,\mathrm{d}y \ge c
$$ 
holds for some positive $c$.

Let $\mathcal{B}=\mathcal{B}(p_0, R)\subset \omega$ be a ball such that $\int_{\mathcal{B}} B(x, y)\,\mathrm{d}x\,\mathrm{d}y$ is not zero. We are going to use the technics developed in \cite{EKP16}. By diamagnetic inequality we get
\begin{align}\nonumber
\int_\omega |i \nabla u+A u|^2\,\mathrm{d}x\,\mathrm{d}y &\ge \frac{1}{2} \int_\omega |i \nabla u+A u|^2\,\mathrm{d}x\,\mathrm{d}y+\frac{1}{2} \int_{\mathcal{B}} |i \nabla u+A u|^2\,\mathrm{d}x\,\mathrm{d}y \\ \label{diamagn.} &\ge  \frac{1}{2} \int_\omega |\nabla |u||^2\,\mathrm{d}x\,\mathrm{d}y+\frac{1}{2} \int_{\mathcal{B}} |i \nabla u+A u|^2\,\mathrm{d}x\,\mathrm{d}y.
\end{align}
According to Lemma\,3.1 in \cite{EK05} there is a constant $c_0=c_0(B)>0$ such that
$$
\int_{\mathcal{B}} |i \nabla u+A u|^2\,\mathrm{d}x\,\mathrm{d}y\ge c_0 \int_{\mathcal{B}} |u|^2\,\mathrm{d}x\,\mathrm{d}y.
$$
The estimate \eqref{diamagn.} implies that for any $u$ satisfying $\|u\|=1$ we have
\begin{align}\nonumber
& \int_\omega |i \nabla u+A u|^2\,\mathrm{d}x\,\mathrm{d}y \ge \frac{1}{2} \int_\omega |\nabla |u||^2\,\mathrm{d}x\,\mathrm{d}y+\frac{c_0}{2} \int_{\mathcal{B}} |u|^2\,\mathrm{d}x\,\mathrm{d}y\\ \label{estimate1} & \quad \ge \mathrm{min}\big\{\textstyle{\frac12}, \textstyle{\frac12}c_0\big\} \underset{ \underset{w\ge0, \|w\|=1}{w\in \mathcal{H}^1(\omega),}}{\mathrm{inf}} \left( \int_\omega |\nabla w|^2\,\mathrm{d}x\,\mathrm{d}y+ \int_{\mathcal{B}} |w|^2\,\mathrm{d}x\,\mathrm{d}y\right)
\end{align}
Let us estimate the right-hand side of \eqref{estimate1}. For any $w\in \mathcal{H}^1(\omega)$ with the unit $L^2$ norm we use the representation
$$
w=\beta \varphi+ f,
$$
where $\varphi=\frac{1}{\sqrt{\mathrm{vol} (\omega)}}$ is the positive normalized ground state eigenfunction of the Neumann Laplacian on $\omega$, $f$ is orthogonal to $\varphi$, and
\begin{equation}\label{condition}
|\beta|^2+\|f\|_{L^2(\omega)}^2=1.
\end{equation}
This choice means that
$$
\int_\omega |\nabla f|^2\,\mathrm{d}x\,\mathrm{d}y\ge \lambda_2\int_\omega |f|^2\,\mathrm{d}x\,\mathrm{d}y,
$$
where $\lambda_2$ is the second eigenvalue of the Neumann Lapacian on $\omega$, and consequently
\begin{align}\nonumber
\int_\omega |\nabla w|^2\,\mathrm{d}x\,\mathrm{d}y+ \int_{\mathcal{B}} |w|^2\,\mathrm{d}x\,\mathrm{d}y &= \int_\omega |\nabla f|^2\,\mathrm{d}x\,\mathrm{d}y+ \int_{\mathcal{B}} |\beta \varphi+ f |^2\,\mathrm{d}x\,\mathrm{d}y
\\ \label{estimate} &\ge \lambda_2\int_\omega |f|^2\,\mathrm{d}x\,\mathrm{d}y +\int_{\mathcal{B}} |\beta \varphi+ f |^2\,\mathrm{d}x\,\mathrm{d}y.
\end{align}
The functions $\varphi$ and $f$ are orthogonal on $\omega$ but not on $\mathcal{B}$, hence to estimate further the right-hand side of (\ref{estimate}) we have to consider separately three cases. The second integral is certainly not smaller than $\big( \|\beta\varphi\|_{L^2(\mathcal{B})} - \|f\|_{L^2(\mathcal{B})}\big)^2$, hence if
$$
|\beta|^2 \int_{\mathcal{B}} |\varphi|^2\,\mathrm{d}x\,\mathrm{d}y\ge 4 \int_{\mathcal{B}} |f|^2\,\mathrm{d}x\,\mathrm{d}y
$$
the expression in question can be estimated as
$$
\int_{\mathcal{B}} |\beta \varphi+ f |^2\,\mathrm{d}x\,\mathrm{d}y \ge \lambda_2\int_\omega |f|^2\,\mathrm{d}x\,\mathrm{d}y+\frac{ |\beta|^2}{4} \int_{\mathcal{B}} |\varphi|^2\,\mathrm{d}x\,\mathrm{d}y,
$$
and similarly for
$$
\int_{\mathcal{B}} |f|^2\,\mathrm{d}x\,\mathrm{d}y\ge 4 |\beta|^2 \int_{\mathcal{B}}|\varphi|^2\,\mathrm{d}x\,\mathrm{d}y
$$
the lower bound is
$$
\int_{\mathcal{B}} |\beta \varphi+ f |^2\,\mathrm{d}x\,\mathrm{d}y \ge \lambda_2 \int_\omega |f|^2\,\mathrm{d}x\,\mathrm{d}y+|\beta|^2 \int_{\mathcal{B}} |\varphi|^2\,\mathrm{d}x\,\mathrm{d}y.
$$
In the remaining case when
$$
\frac{1}{4} \int_{\mathcal{B}} |f|^2\,\mathrm{d}x\,\mathrm{d}y< |\beta|^2 \int_{\mathcal{B}} |\varphi|^2\,\mathrm{d}x\,\mathrm{d}y<4 \int_{\mathcal{B}} |f|^2\,\mathrm{d}x\,\mathrm{d}y
$$
we just know that the squared norm difference is nonnegative. Instead, we then split the first term on the right-hand side of (\ref{estimate}) and estimate one half of it from below from the last inequality obtaining
$$
\frac{\lambda_2}{2}\int_\omega |f|^2\,\mathrm{d}x\,\mathrm{d}y+\frac{ |\beta|^2 \lambda_2}{8} \int_{\mathcal{B}} |\varphi|^2\,\mathrm{d}x\,\mathrm{d}y.
$$
Putting these bounds together we get
$$
\int_\omega |\nabla w|^2\,\mathrm{d}x\,\mathrm{d}y+ \int_{\mathcal{B}} |w|^2\,\mathrm{d}x\,\mathrm{d}y\ge \frac{\lambda_2}{2}\int_\omega |f|^2\,\mathrm{d}x\,\mathrm{d}y + \frac{|\beta|^2}{4} \min\left\{1, \frac{\lambda_2}{2}\right\} \frac{\mathrm{vol}(\mathcal{B})}{\mathrm{vol}(\omega)}.
$$
If $\beta\ge 1/2$ the right-hand side of the last inequality does not exceed $\frac{1}{16} \min\left\{1, \frac{\lambda_2}{2}\right\} \frac{\mathrm{vol}(\mathcal{B})}{\mathrm{vol}(\omega)}$, in the opposite case relation \eqref{condition} shows that the lower bound is $\frac38\lambda_2 > \frac{1}{16} \min\left\{1, \frac{\lambda_2}{2}\right\} \frac{\mathrm{vol}(\mathcal{B})}{\mathrm{vol}(\omega)}$. Combining finally this conclusion with relation  \eqref{estimate1} we arrive at the bound
\begin{equation}\label{lower bound}
\mathrm{inf} \sigma(-\Delta_N^\omega) \ge\frac{\mathrm{vol}(\mathcal{B})}{16 \mathrm{vol}(\omega)} \min\left\{1, \frac{\lambda_2}{2}\right\}\min\left\{1, \frac{c_0}{2}\right\},
\end{equation}
which proves the lemma.
\end{proof}

\medskip

We also need the following simple estimate:
\begin{lemma}\label{2lemma}
Let $I\subset \mathbb{R}$ be an interval and denote by $|I|$ its length, then for any $g\in \mathcal{H}^1(I)$ and all $x\in I$ we have
$$
|g(x)|^2\le 2|I| \int_I |g'(t)|^2\,\mathrm{d}t +\frac{2}{|I|} \int_I |g(t)|^2\, \mathrm{d}t
$$
\end{lemma}
\begin{proof}
One can check easily that there exists a point $x_0\in I$ such that
$$
|g(x_0)|\le \frac{1}{\sqrt{|I|}} \sqrt{\int_I |g(t)|^2 dt}\,;
$$
this fact in combination with Schwarz inequality yields
$$|g(x)|^2=\left|\int_{x_0}^x g'(t)\,\mathrm{d}- g(x_0)\right|^2\le 2|I| \int_I |g'(t)|^2\,\mathrm{d} +\frac{2}{|I|} \int_I |g(t)|^2\,\mathrm{d}.$$
which is what we have set up to prove.
\end{proof}

\medskip

After this preliminary let us return to proof of the theorem. Applying Lemma~\ref{2lemma} to $u(x, y)$ on $\Omega$ for a fixed $x$ one gets
\begin{equation}\label{u}
|u(x, y)|\le 2C \int_{-a}^a (|u_t(x, t)|^2+|u(x, t)|^2)\,dt,\quad y\in (-a, a),
\end{equation}
where
\begin{equation}\label{constant}
C:= \mathrm{max}\Big\{2a, \frac{1}{2a}\Big\}.
\end{equation}
In what follows we use again the notation $\Gamma_\Omega=\{x, \gamma(x)\}_{x\in (-a, a)}$. We have
$$
\int_{\Gamma_\Omega} |u|^2\, \mathrm{d}\mu=\int_{-a}^a |u(x, \gamma(x))|^2 \sqrt{1+\gamma'(x)^2}\,\mathrm{d}x
$$
and relation\eqref{u} allows to estimate the above expression as follows,
$$
\int_{\Gamma_\Omega} |u^2|\,\mathrm{d}\mu \le C\int_{-a}^a \sqrt{1+\gamma'(x)^2} \int_{-a}^a \left(|u_y(x, y)|^2+|u(x, y)|^2\right)\,\mathrm{d}y\,\mathrm{d}x,
$$
with the constant $C$ defined in (\ref{constant}); in this way we arrive to the bound
\begin{equation}\label{trace}
\int_{\Gamma_\Omega} |u|^2\,\mathrm{d}\mu \le C \sqrt{1+ \|\gamma'\|^2_\infty}  \int_\Omega \left(|u_y(x, y)|^2+|u(x, y)|^2\right)\,\mathrm{d}x\,\mathrm{d}y.
 \end{equation}
Using once more Landau gauge for the vector potential, $A=\left(-\int_0^y B(x, t)\,\mathrm{d}t, 0\right)$, we infer from \eqref{trace} that
\begin{equation}\label{magnetic}
\int_{\Gamma_\Omega} |u|^2\,d \mu \le C \sqrt{1+\|\gamma'\|^2_\infty}\left(\int_\Omega |i \nabla u+A u|^2\,\mathrm{d}x\,\mathrm{d}y+ \int_\Omega |u|^2\,\mathrm{d}x\,\mathrm{d}y\right).
\end{equation}

Next we use Lemma~\ref{1lemma} by which the ground state eigenvalue of the magnetic Neumann Laplacian on $\Omega$ is not smaller than some $\kappa=\kappa_{\Omega, B}>0$. The inequality (\ref{magnetic}) implies
\begin{align}
\nonumber \int_\Omega |i \nabla u &+ A u|^2\,\mathrm{d}x\,\mathrm{d}y- \alpha \int_{\Gamma_\Omega}  |u^2|\,\mathrm{d}\mu \\ \nonumber & \ge
\left(1-C \alpha\sqrt{1+\|\gamma'\|^2_\infty}\right) \int_\Omega |i \nabla u+A u|^2\,\mathrm{d}x\,\mathrm{d}y-
C \alpha\sqrt{1+\|\gamma'\|^2_\infty} \int_\Omega |u|^2\,\mathrm{d}x\,\mathrm{d}y \\ \nonumber & \ge \kappa\left(1- C\alpha\sqrt{1+\|\gamma'\|^2_\infty}\right) \int_\Omega |u|^2\,\mathrm{d}x\,\mathrm{d}y- C\alpha \sqrt{1+\|\gamma'\|^2_\infty} \int_\Omega |u|^2\,\mathrm{d}x\,\mathrm{d}y
\\ & \label{magn trace} = \left(\kappa- C \alpha  (1+\kappa)\sqrt{1+\|\gamma'\|^2_\infty}\right)\int_\Omega |u|^2\,\mathrm{d}x\,\mathrm{d}y,\end{align}
hence by choosing
\begin{equation}\label{alpha}
\alpha \sqrt{1+\|\gamma'\|^2_\infty}\le \frac{\kappa}{(\kappa+1) C}
\end{equation}
we achieve that the right-hand side of \eqref{magn trace} becomes non negative. In combination with \eqref{final} this means that the spectrum below $-\frac14{\alpha^2}$ is empty.

To finish the proof we have to specify the range of the values of $\alpha$ compatible with \eqref{alpha}. We note that $\|\gamma'\|_\infty\ge 1$. Indeed, was it not the case we get a contradiction because the identity
$$
\gamma(a)=\int_{-a}^a\gamma'(t)\,\mathrm{d}t+ \gamma(-a)
$$
would lead to a contradiction with $\gamma(\pm a )=\pm a$. Hence $\sqrt{1+\|\gamma'\|^2_\infty}\ge \sqrt{2}$ and the condition \eqref{alpha} may be satisfied only if
\begin{equation}\label{alpha0}
\alpha \le \alpha_0:=\frac{\kappa}{\sqrt{2} (\kappa+1) C},
\end{equation}
where $C$ is the number given by \eqref{constant}. For a fixed $\alpha\in(0,\alpha_0]$ the condition \eqref{alpha} then reads $\sqrt{1+\|\gamma'\|^2_\infty}\le \frac{\sqrt{2}\, \alpha_0}{\alpha}$ which shows how much the curve $\Gamma$ may depart from the straight line without giving rise to bound states,
$$
\|\gamma'\|_\infty\le \frac{2 \alpha_0^2}{\alpha^2}-1.
$$
This completes the proof of Theorem~\ref{th.2}.

\section{Concluding remarks}
\label{proofs}
\setcounter{equation}{0}

Let us add a few comments. First of all, for weak magnetic fields the value $\alpha_0$ of the critical coupling is expected to be small. By \eqref{alpha0} its behavior depends on the principal eigenvalue of the magnetic Neumann Laplacian on $\Omega$ which is expected to go zero as $B\to 0\:$; one could use, for instance, \cite{CEIS18}, with a bit of extra work considering first the operator on a region $\Omega'\supset\Omega$ with a smooth boundary. Furthermore, for a fixed $\gamma$ one can always achieve that the discrete spectrum is empty by choosing $\alpha$ small enough. In this respect it is useful to mention that the analogy with Dirichlet waveguides that we mentioned in the introduction as an inspiration is far from complete: a counterpart to the coupling strength $\alpha$ could be, in a sense, the inverse of the waveguide width but the latter cannot take arbitrary values, for smooth waveguides at least.

On the other hand, our geometric assumptions are rather restrictive and it would be useful to find an extension of the present result, at least to curves $\Gamma$ straight outside a compact region but not necessarily coming from a local deformation of a straight line.

\bigskip

\subsection*{Acknowledgment}

The authors are obliged to Hynek Kova\v{r}\'{\i}k for useful comments. The work of P.E. was in part supported by the European Union within the project CZ.02.1.01/0.0/0.0/16 019/0000778.

\end{document}